\documentclass{amsart}

\usepackage{hyperref}
\usepackage{cleveref}
\usepackage{amsfonts}
\usepackage{url}
\usepackage{enumerate}
\usepackage{amssymb}
\usepackage{amsmath}
\usepackage{xy}
\input xy
\xyoption{all}

\usepackage{amsthm}

\usepackage{tikz}

\newcommand{\floor}[1]{\left\lfloor #1 \right\rfloor}
\newcommand{\ceil}[1]{\left\lceil#1 \right\rceil}

\renewcommand{\hom}{\mathrm{Hom}}
\newtheorem{lemma}{Lemma}[section]

\newtheorem{corollary}[lemma]{Corollary}

\newtheorem{theorem}[lemma]{Theorem}

\newtheorem{proposition}[lemma]{Proposition}

\theoremstyle{definition}
\usepackage{verbatim}
\newtheorem{definition}[lemma]{Definition}
\newtheorem*{remark}{Remark}
\newtheorem{example}[lemma]{Example}

\begin{document}

\renewcommand{\thesubsection}{\arabic{subsection}}
\title{Torsion exponents in stable homotopy and the Hurewicz homomorphism}
\author{Akhil Mathew}
\date{\today}
\email{amathew@math.berkeley.edu}
\address{University of California, Berkeley CA 94720}

\begin{abstract} We give estimates for the torsion in the Postnikov sections
$\tau_{[1, n]} S^0$ of
the sphere spectrum, and show that the $p$-localization is annihilated by
$p^{n/(2p-2) + O(1)}$. 
This leads to explicit bounds on the exponents of the kernel and cokernel of the
Hurewicz map $\pi_*(X) \to H_*(X; \mathbb{Z})$ for a connective spectrum $X$.
Such bounds were first considered by Arlettaz, although our estimates are
tighter and we prove that they are the best possible up to a constant factor. 
 As applications, we sharpen existing bounds on the orders of
$k$-invariants in a connective spectrum,  sharpen bounds on the unstable
Hurewicz map of an infinite loop space, and prove an exponent theorem for the
equivariant stable stems. 
\end{abstract}

\maketitle

\section{Introduction}

Let $X$ be a spectrum. 
Then there is a natural map
(the Hurewicz map)
of graded abelian groups
\[ \pi_*(X) \to H_*(X; \mathbb{Z})  \]
which is an isomorphism rationally. In general, this is the best that one can
say. For instance, given an element 
$x \in \pi_n(X)$
annihilated by the Hurewicz map, we know that $x$ is torsion, but we cannot a
priori give an integer $m$ such that $mx = 0$. For example, if $K$ denotes
periodic complex $K$-theory, then $K/p^k$ has trivial homology for each $k$, but it
contains elements in homotopy of order $p^k$.

If, however, $X$ is connective, then 
one can do better. 
For instance, the Hurewicz theorem states in this case that the map $\pi_0(X)
\to H_0(X; \mathbb{Z})$ is an isomorphism. 
The map $\pi_1(X) \to H_1(X; \mathbb{Z})$ need not be an isomorphism, but it
is surjective and any
element in the kernel must be annihilated by 2. 
There is a formal argument that in
any degree, ``universal'' bounds must exist. 

\begin{proposition} 
\label{existbounds}
There exists a function $M\colon \mathbb{Z}_{\geq 0} \to \mathbb{Z}_{>0}$
with the following property: 
if $X$ is any connective spectrum, then the kernel and cokernel of the Hurewicz
map $\pi_n(X) \to H_n(X; \mathbb{Z})$ are annihilated by $M(n)$.
\end{proposition} 
\begin{proof} 
We consider the case of the kernel; the other case is similar. Suppose there
existed no such function. Then, there exists an integer $n$ and
connective spectra $X_1,
X_2, \dots $ together with elements $x_i \in \pi_n(X_i)$ for each $i$ such that:
\begin{enumerate}[(a)]
\item $x_i$ is  in the kernel of the Hurewicz map (and thus torsion).
\item The orders of the $x_i$ are unbounded.
\end{enumerate}
In this case, we can form a connective spectrum $X = \prod_{i = 1}^\infty X_i$.
Since homology commutes with arbitrary products \emph{for connective spectra},
as $H \mathbb{Z}$ can be given a cell decomposition with finitely many cells in
each degree (see \cite[Thm. 15.2, part III]{Adams}),
it follows that we obtain an element $x  = (x_i)_{i \geq 1}  \in \pi_n(X) =
\prod_{i \geq 1} \pi_n(X_i)$ which is annihilated by the Hurewicz map. However, $x$
cannot be torsion since the orders of the $x_i$ are unbounded.
\end{proof}

We note that the above argument is very general. 
For instance, it shows that the nilpotence theorem \cite{DHS} implies that there exists 
a universal function $P(n)\colon \mathbb{Z}_{\geq 0} \to \mathbb{Z}_{>0}$ such that
if $R$ is a connective ring spectrum and $x \in \pi_n(R)$ is annihilated by the
$MU$-Hurewicz map, then $x^{P(n)} =0$. The determination of the best
possible function $P(n)$
is closely related to the questions raised by Hopkins in \cite{Raveneltalk}. 

\Cref{existbounds} appears in \cite{AMoore}, where an upper bound for the universal 
function $M(n)$ is established (although the above argument may be older).
\begin{theorem}[{Arlettaz \cite[Thm. 4.1]{AMoore}}] If $X$ is any
connective spectrum, then the kernel of $\pi_n(X) \to H_n(X; \mathbb{Z})$ is
annihilated by  
$\rho_1 \dots \rho_n$ where $\rho_i$ is the smallest positive integer
that annihilates the torsion group
$\pi_i(S^0)$. 
The cokernel is annihilated by 
$\rho_1 \dots \rho_{n-1}$.
\end{theorem} 

Different variants of this result have appeared in \cite{Arlettazfirst,
generalizedBoardman}, and this result has also been discussed in
\cite{beilinson}. 
The purpose of this note is to find the best possible bounds for these torsion 
exponents, up to small constants. We will do so at each prime $p$. 
In particular, we prove: 

\begin{theorem} 
\label{ourHurbound}
Let $X$ be a connective spectrum and let $n > 0$. Then: 
\begin{enumerate}[(a)]
\item The 2-exponent of the kernel of the Hurewicz map $\pi_n(X) \to H_n(X;
\mathbb{Z})$ is at most $\ceil{\frac{n}{2}} + 3$: that is,
$2^{\ceil{\frac{n}{2}} + 3}$ annihilates the 2-part of the kernel.  
\item  If $p$ is an odd prime,  
the $p$-exponent of the kernel of the Hurewicz map $\pi_n(X) \to H_n(X;
\mathbb{Z})$ is at most $\ceil{ \frac{n + 3}{2p-2} }  + 1$.
\item The 2-exponent of the cokernel of the Hurewicz map is at most 
$\ceil{\frac{n-1}{2}} + 3$.
\item  If $p$ is an odd prime, 
the $p$-exponent of the cokernel of the Hurewicz map is at most $\ceil{ \frac{n +
2}{2p-2} }  + 1$. 
\end{enumerate}
\end{theorem}

We will also show that these bounds are close to being the best possible. 
\begin{proposition} \label{hurbestpossible}
\begin{enumerate}[(a)]
\item For each $r$, there exists a connective 2-local spectrum $X$ and an
element $x \in \pi_{2r-1}(X)$ in the kernel of the Hurewicz map such that the order
of $x$ is at least $2^{r-1}$. 
\item Let $p$ be an odd prime. For each $r$, there exists a connective $p$-local 
spectrum $X$ and an element $x \in \pi_{(2p-2)r + 1}(X)$ annihilated by the
Hurewicz map such that the order of $x$ is at least $p^{r}$. 
\end{enumerate}
\end{proposition} 

Our strategy in proving \Cref{ourHurbound} is to translate the above 
question into one about the Postnikov sections $\tau_{[1, n]} S^0$ and their
exponents in the homotopy category of spectra (rather than the exponents of
some algebraic invariant). We shall use a classical
technique with vanishing lines to show that, at a prime $p$, the $\tau_{[1, n]} S^0$
are annihilated by $p^{n/(2p-2) + O(1)}$. 
This, combined with a bit of diagram-chasing, will imply the upper bound of 
\Cref{ourHurbound}. The lower bounds will follow from explicit examples.

Finally, we show that these methods have additional applications and that the
precise order of the $n$-truncations $\tau_{[1, n]}S^0$ play an important
role in several settings. For instance,
we sharpen bounds of Arlettaz \cite{arlettazkinv} on the orders of the $k$-invariants
of a spectrum (\Cref{ourboundkinv}), improve and make explicit half of a result of Beilinson \cite{beilinson} on
the (unstable) Hurewicz map $\pi_n(X) \to H_n(X; \mathbb{Z})$ for $X$ an
infinite loop space (\Cref{hurmapifl}), and prove an exponent theorem for the equivariant stable
stems (\Cref{equivexponent}).

We also obtain as a consequence the following result. 
\begin{theorem} 
\label{torsionorder}
Let $p$ be a prime number. 
Let $X$ be a spectrum with homotopy groups concentrated in 
degrees $[a, b]$. Suppose each $\pi_i(X)$ is annihilated by $p^k$. Then $p^{k + 
\frac{b-a}{p-1} + 8}$ annihilates $X$ (\Cref{annihilated} below).  
\end{theorem} 

We have not tried to make the bounds in \Cref{torsionorder} as sharp as
possible since we suspect that our techniques are not sharp to begin with.

\subsection*{Notation}
In this paper, for a spectrum $X$, we will write $\tau_{[a, b]} X$ to denote
the Postnikov section of $X$ with homotopy groups in the range $[a, b]$, i.e.,
$\tau_{\geq b} \tau_{\leq a} X$.
Given spectra $X, Y$, we will let $\hom(X, Y)$ denote the function spectrum
from $X$ into $Y$, so that $\pi_0 \hom(X, Y)$
denotes homotopy classes of maps $X \to Y$.

\subsection*{Acknowledgments}
I would like to thank Mike Hopkins and Haynes Miller, from whom (and whose papers) I learned 
many of the ideas used here.  
I would also like to thank Peter May for several helpful comments and Dustin Clausen for pointing me to \cite{beilinson}. 
The author was supported by the NSF Graduate
Fellowship under grant  DGE-110640.

\section{Definitions}

Let $\mathcal{C}$ be a triangulated category. 
We recall:

\begin{definition} \label{annihilated}
Let $X \in \mathcal{C}$ be an object. We will say that $X$ is \emph{annihilated
by $n \in \mathbb{Z}_{>0}$}
if $n
\mathrm{id}_X \in  \hom_{\mathcal{C}}(X, X) $ is equal to zero.
We let $\mathrm{exp}(X)$ denote the minimal $n$ (or $\infty$ if no such exists) such that $n$
annihilates $X$.
\end{definition}

If $\mathcal{D}$ is any additive category and $F\colon \mathcal{C} \to \mathcal{D}$
any additive functor, then if $X \in \mathcal{C} $ is annihilated by $n$, then
$F(X) \in \mathcal{D}$ has $n \mathrm{id}_{F(X)} =0 $ too. 
Here are several important examples of this phenomenon. 

\begin{example}
\label{cohomfunctor}
Given any (co)homological functor $F\colon \mathcal{C} \to \mathrm{Ab}$,
the value of $F$ on an object annihilated by $n$ is  a torsion abelian group 
of exponent at most $n$.
For instance, if $X$ is a spectrum annihilated by $n$, then the homotopy groups
of $X$ all have exponent at most $n$. 
\end{example}
\begin{example} 
\label{truncationtor}
Suppose $\mathcal{C} $ has a $t$-structure, so that we can construct
truncation functors $\tau_{\leq k} \colon \mathcal{C} \to \mathcal{C}$ for $k
\in \mathbb{Z}$. Let $X \in \mathcal{C}$ be any
object. Then, for any $k$, $\mathrm{exp}(\tau_{\leq k} X) \mid \mathrm{exp}(X)$.
\end{example} 
\begin{example} 
\label{smashexponent}
Suppose $\mathcal{C}$ has a compatible monoidal structure $\wedge$. Then if
$X, Y \in \mathcal{C}$, we have $\mathrm{exp} (X \wedge Y) \mid 
\mathrm{gcd}\left(\mathrm{exp}(X), \mathrm{exp}(Y)\right)$.
\end{example}

Next, we note that such torsion questions can be reduced to local ones at each prime $p$, and it
will be therefore convenient to have the following notation.
\begin{definition} 
Given $X \in \mathcal{C}$, we define $\exp_p(X)$ to be the minimal integer $n
\geq 0$
(or $\infty$ if none such exists) such that $p^n \mathrm{id}_X = 0 $ in the
group $\hom_{\mathcal{C}}(X, X)_{(p)}$.
For a torsion abelian group $A$, we will also use the notation $\exp_p(A)$ in
this manner. 
\end{definition}

\begin{proposition} 
\label{expcofiber}
Let $X' \to X \to X''$ be a cofiber sequence in $\mathcal{C}$. Suppose $X'$ is
annihilated by $m$ and $X''$ is annihilated by $n$. Then $X$ is annihilated by
$mn$. Equivalently, $\exp_p(X) \leq \exp_p(X') +\exp_p(X'')$ for each prime $p$.
\end{proposition}\begin{proof} 

We have an exact sequence
of abelian groups
\[ \hom_{\mathcal{C}}(X, X' ) \to \hom_{\mathcal{C}}(X, X) \to
\hom_{\mathcal{C}}(X, X'').  \]
If $X'$ (resp. $X''$) is annihilated by $m$ (resp. annihilated by $n$), then it
follows that groups on the edges of the above exact sequence are of exponents 
dividing $m$ and $n$, respectively. It follows that $\hom_{\mathcal{C}}(X, X)$ is
annihilated by $mn$, and in particular the identity map $\mathrm{id}_X \in
\hom_{\mathcal{C}}(X, X)$ is annihilated by $mn$. 
\end{proof} 
\begin{corollary} 
\label{firstbound}
Let $X$ be a spectrum with homotopy groups concentrated in degrees 
$[m, n]$ for $m , n \in \mathbb{Z}$. Suppose for each $i \in [m, n]$, we have
an integer $e_i > 0$ with $e_i \pi_i(X) = 0$. Then $\mathrm{exp}(X) \mid \prod_{i =
m}^n e_i$.
\end{corollary}

The main purpose of this paper is to determine the behavior of the function $\exp_p(
\tau_{[1, n]} S^0)$ as $n$ varies. \Cref{firstbound} gives the bound that 
$\exp_p(
\tau_{[1, n]} S^0)$ is at most the sum  of the exponents of the
torsion abelian groups $\pi_i(S^0)_{(p)}$ for $1 \leq i \leq n$. We will give a
stronger upper
bound for this function, and show that it is essentially optimal.

\begin{theorem}[Main theorem]
\label{maintheorem}
\begin{enumerate}[(a)]
\item  
Let $p  = 2$. Then: 
\begin{equation}\floor{\frac{n-1}{2}}  \leq \exp_2( \tau_{[1, n]} S^0)  \leq
\ceil{\frac{n}{2}} + 3. \end{equation}
\item
Let $p$ be odd. Then:
\begin{equation}
\floor{ \frac{n - 1}{2p-2}}  \leq 
\exp_p( \tau_{[1, n]} S^0)
\leq 
\ceil{ \frac{n + 3}{2p-2} }  + 1
\end{equation}

\end{enumerate}
\end{theorem} 

The upper bounds will be proved in \Cref{upperbounds} below, and the lower
bounds will be proved in \Cref{lowerbound2} and \Cref{lowerboundp}.
They include as a special case estimates on the exponents on the \emph{homotopy groups}
of $S^0$, which were classically known (and in fact our method is a refinement
of the proof of those estimates). 
Note that the exponents in the \emph{unstable} homotopy 
groups have been studied extensively, including the precise 
determination 
at odd primes \cite{CMN}, and that the method of using the Adams spectral
sequence to obtain such quantitative bounds has also been used by Henn
\cite{henn}. 

\section{Upper bounds}

\newcommand{\A}{\mathcal{A}}

Let $p$ be a prime number. Let $\A_p$ denote the mod $p$ Steenrod algebra; it
is a graded algebra. Recall that if $X$ is a spectrum, then the mod $p$
cohomology $H^*(X; \mathbb{F}_p)$ is a graded  module over $\A_p$. Our approach
to the upper bounds will be based on vanishing lines in the cohomology. 

\begin{definition} 
Given a nonnegatively graded $\A_p$-module $M$, we will say that a function $f\colon
\mathbb{Z}_{\geq 0} \to \mathbb{Z}_{\geq 0}$ is a \emph{vanishing function}
for $M$ if 
for all $s,t \in \mathbb{Z}_{\geq 0}$,
\[ \mathrm{Ext}^{s,t}_{\mathcal{A}_p}(M, \mathbb{F}_2) = 0 \quad \text{ if }   t < f(s) .  \]
Recall here that $s$ is the homological degree, and $t$ is the grading. 
\end{definition}

Our main technical result is the following: 

\begin{proposition} 
\label{vanishlinetorsion}
Suppose $X$ is a connective spectrum such that each $\pi_i(X)$ is a finite
$p$-group. Suppose the $\A_p$-module $H^*(X; \mathbb{F}_p)$
has a vanishing function $f$. 
Let $n$ be an integer and let $m$ be an integer such that $f(m) - m >
n$. Then $\exp_p( \tau_{[0, n]} X) \leq m$.
\end{proposition} 
\begin{proof} 
Choose a minimal resolution (see, e.g., \cite[Def. 9.3]{usersguide}) of
$H^*(X; \mathbb{F}_p)$ by free, graded $\A_p$-modules
\begin{equation} \label{algres} \dots \to P_1 \to P_0 \to H^*(X; \mathbb{F}_p)
\to 0.  \end{equation}
In this case, we have $\mathrm{Ext}^{s,t}(H^*(X; \mathbb{F}_p), \mathbb{F}_p)
\simeq
\hom_{\A_p}( P_s, \Sigma^t \mathbb{F}_p)$ by \cite[Prop. 9.4]{usersguide}. 
That is, the free generators of the $P_s$ give precisely a basis for
$\mathrm{Ext}^{s, \ast}( H^*(X; \mathbb{F}_p); \mathbb{F}_p)$.

We can realize the resolution \eqref{algres} topologically via an Adams
resolution (cf., e.g., \cite[\S 9.3]{usersguide}). 
That is, we can find (working by induction) a tower
of spectra,
\begin{equation} \label{res}  \xymatrix{
\vdots \ar[d]  \\
F_2 X \ar[d] \ar[r] &  R_2 \\
F_1 X \ar[d]  \ar[r] &  R_1  \\
F_0 X = X \ar[r] &  R_0
},\end{equation}
such that:
\begin{enumerate}[(a)]
\item  
Each $R_i$ is a wedge of copies of shifts of $H \mathbb{F}_p$.
\item Each triangle $F_{i+1} X \to F_i X \to R_i$ is a cofiber sequence.
\item The sequence of spectra
\[ X \to R_0 \to \Sigma R_1 \to \Sigma^2 R_2 \to \dots   \]
realizes on cohomology the complex
\eqref{algres}.

\end{enumerate}
As a result, we find inductively that 
\[ H^*(F_i X; \mathbb{F}_p) \simeq \Sigma^{-i}\mathrm{im}(P_{i}  \to P_{i-1} ).    \]
Now the graded $\mathcal{A}_p$-module $P_i$ is concentrated in degrees 
$f(i)$ and up, by hypothesis and minimality. In particular, 
it follows that 
$F_i X$ is $(f(i) - i)$-connective. 
It follows, in particular, that the map
\[ X \to \mathrm{cofib}(F_i X \to X)  \]
is an isomorphism on homotopy groups below $f(i) - i$. 

Finally, we observe that the cofiber of each $F_i X \to F_{i-1} X$ is
annihilated by $p$ as it is a wedge of shifts of $H \mathbb{F}_p$. It follows by the octahedral axiom of triangulated
categories, induction on $i$, and \Cref{expcofiber} that the cofiber of $F_i X \to F_0 X = X$ is annihilated by $p^i$. 
Taking $i = m$, we get the claim since $\tau_{\leq n}X \simeq \tau_{\leq n}( \mathrm{cofib}(F_m X \to
X))$ is therefore annihilated by $p^m$ by \Cref{truncationtor}. 
\end{proof} 

Since $\mathcal{A}_p$ is a \emph{connected} graded algebra, it follows easily
(via a minimal resolution)
that if $M$ is a connected graded $\mathcal{A}_p$-module, then
$\mathrm{Ext}^{s,t}(M, \mathbb{F}_p) = 0$ if $t < s$. 
Of course, this bound is too weak to help with \Cref{vanishlinetorsion}.
In fact, an integer $m$ satisfying the desired conditions will not 
exist if we use this bound. 

We now specialize to the case of interest.
Consider $\tau_{\geq 1} S^0= \tau_{[1, \infty]}
S^0$. It fits into a cofiber sequence 
\[  S^0 \to H \mathbb{Z}
\to \Sigma \tau_{\geq 1} S^0
  ,\]
  which leads to an exact sequence
  \[ 0 \to H^*( \Sigma \tau_{\geq 1} S^0; \mathbb{F}_p ) \to 
  H^*( H \mathbb{Z}; \mathbb{F}_p) \to H^*(S^0; \mathbb{F}_p) \to 0.
  \]
Now we know that (by the change-of-rings theorem \cite[Fact 3, p.
438]{usersguide}) $
\mathrm{Ext}_{\A_p}^{s,t}( 
 H^*( H \mathbb{Z}; \mathbb{F}_p); \mathbb{F}_p) $ vanishes unless $s = t$, and
 is one-dimensional if $s = t$; in this case it maps isomorphically to
 $\mathrm{Ext}^{s,s} _{\A_p}( \mathbb{F}_p, \mathbb{F}_p)$. 
 It follows: 
 \begin{equation} \label{dimshift} \mathrm{Ext}^{s,t}_{\A_p}(H^*(\tau_{\geq 1} S^0; \mathbb{F}_p ); \mathbb{F}_p ) = 
\begin{cases} 
 \mathrm{Ext}_{\A_p}^{s-1,t-1}(\mathbb{F}_p; \mathbb{F}_p ) & s \neq t \\
 0 & \text{if } s = t
 \end{cases} 
 \end{equation}

We will need certain classical facts, due to Adams 
\cite{adamslecture} at $p = 2$
and Liulevicius \cite{liulevicius} for $p >2$, about
vanshing lines 
in the classical Adams spectral sequence. A convenient reference is
\cite{usersguide}. 
\begin{proposition}[{\cite[Thm. 9.43]{usersguide}}]
\label{vanishline}
\begin{enumerate}[(a)]
\item 
$\mathrm{Ext}_{\A_2}^{s,t}(\mathbb{F}_2, \mathbb{F}_2) = 0$ for $0 < s < t < 3s -
3$.
\item 
$\mathrm{Ext}_{\A_p}^{s,t}(\mathbb{F}_p, \mathbb{F}_p) = 0$
for $0 < s < t < (2p-1) s - 2$.
\end{enumerate}
\end{proposition} 
Note also that $\mathrm{Ext}^{s,t}_{\A_p}(\mathbb{F}_p, \mathbb{F}_p) = 0$ for
$t < s$. As a result, one finds that the cohomology of $\tau_{\geq 1} S^0$,
when displayed with Adams indexing with $t-s$ on the $x$-axis and $s$ on the $y$-axis, vanishes
above a line with slope $\frac{1}{2p-2}$.

Finally, we can prove our upper bounds.
\begin{proposition} 
\label{upperbounds}
\begin{enumerate}[(a)]
\item 
$\exp_2( \tau_{[1, n]} S^0)  \leq
\ceil{\frac{n}{2}} + 3$.
\item 
For $p$ odd, 
$\exp_p( \tau_{[1 , n]} S^0) \leq \ceil{\frac{n + 3}{2p-2}} + 1$.
\end{enumerate}
\end{proposition} 
\begin{proof}

This is now a consequence of the preceding discussion. We just need to put
things together.

At the prime 2, it follows 
from \Cref{vanishline} and \eqref{dimshift}
that the $\A_2$-module 
$H^*(\tau_{\geq 1} S^0; \mathbb{F}_2 )$ has vanishing function $f(s) = 3s - 5$. 
By \Cref{vanishlinetorsion}, it follows that if $2m - 5 > n$, then $\exp_2(
\tau_{[1, n]}S^0) \leq m$. Choosing $m  = \ceil{\frac{n}{2}} + 3$ gives the
minimal choice. 

At an odd prime, one similarly sees (by \Cref{vanishline} and
\eqref{dimshift}) that $f(s) = (2p-1)s - 2p$ is a vanishing
function. That is, if $(2p-2) m - 2p >
n$, then we have $\exp_p( \tau_{[1, n]} S^0) \leq m$. Rearranging gives the
desired claim. 
\end{proof}

\section{Lower bounds}

The purpose of this section is to prove the lower bounds of \Cref{maintheorem}.
The proof of the lower bounds is completely different from the proof of the
upper bounds. 
Namely, we will write down finite complexes that have homology annihilated by
$p$ but for which the $p$-exponent grows linearly. These complexes are simply
the skeleta of $B \mathbb{Z}/p$. We will show, however, that the
$p$-exponent of the \emph{spectra} grows linearly by 
looking at the complex $K$-theory. 
First, we need a lemma.

\begin{lemma} \label{sufficestotruncate}
Let $X$ be a finite torsion complex with cells in degrees $0$ through $m$.
Then, 
for each $p$,
$\exp_p(X) = \exp_p( \tau_{[0, m]} S^0 \wedge X)$.
\end{lemma} 
\begin{proof} 
Without loss of generality, $X$ is $p$-local. 
We know that $\exp_p(X) \geq 
\exp_p( \tau_{[0, m]} S^0 \wedge X)$ (\Cref{smashexponent}).  Thus, we need to prove the other 
inequality. Let $k = \exp_p(X)$.

Let $\hom(X, X)$ denote the endomorphism ring spectrum of $X$.
The identity map $X \to X$ defines a class in $\pi_0 \hom(X, X)$, which maps
isomorphically to
$\pi_0 \hom(X, \tau_{[0, m]} S^0 \wedge X)$ by the hypothesis on the cells of
$X$. 
Therefore, there exists a class in 
$\pi_0\hom(X, \tau_{[0, m]} S^0 \wedge X)$
of order exactly $p^k$. It follows that 
$\exp_p(\tau_{[0, m]} S^0 \wedge X ) \geq k$ as desired. 
\end{proof}

We are now ready to prove our lower bound at the prime two. 
\begin{proposition} 
\label{lowerbound2}
We have $\mathrm{exp}_2(\tau_{[1, n ] } S^0) \geq  
\floor{(n-1)/2} $.
\end{proposition} 
\begin{proof} 
Since the function 
$n \mapsto \mathrm{exp}_2(\tau_{[1, n ] } S^0)$ is increasing in $n$
(\Cref{truncationtor}), it suffices to assume $n = 2r-1$ is odd.
Consider the space $\mathbb{RP}^{2r}, r \in \mathbb{Z}_{> 0}$ and its reduced
suspension spectrum $\Sigma^\infty \mathbb{RP}^{2r}$, which is 2-power torsion. We know that $\widetilde{K}^0( \mathbb{RP}^{2r}) \simeq \mathbb{Z}/2^{r}$ by
\cite[Prop. 2.7.7]{AtiyahKthy}. It follows that (cf. \Cref{cohomfunctor}) 
\begin{equation} \label{expRP} \exp_2( \Sigma^\infty
\mathbb{RP}^{2r}) \geq r. \end{equation}
Now $\Sigma^\infty \mathbb{RP}^{2r}$ has cells in degrees $1$ to $2r$.
By \Cref{sufficestotruncate}, 
$\exp_2( \tau_{[0, 2r-1]} S^0 \wedge \Sigma^\infty
\mathbb{RP}^{2r}) \geq r$ too.

We have a
cofiber sequence
\[ \tau_{[1, 2r-1]}S^0 \wedge  \Sigma^\infty\mathbb{RP}^{2r} \to 
\tau_{[0, 2r-1]} S^0  \wedge
\Sigma^\infty \mathbb{RP}^{2r} \to H \mathbb{Z} \wedge \Sigma^\infty \mathbb{RP}^{2r}.
\]
The integral homology of $\Sigma^\infty \mathbb{RP}^{2r}$ is annihilated by $2$, so that the
$H \mathbb{Z}$-module spectrum 
$H \mathbb{Z} \wedge \mathbb{RP}^{2r}$ is a wedge of copies of $H \mathbb{Z}/2$
and is thus annihilated by $2$. 
It therefore follows from this cofiber sequence and \Cref{expcofiber} that 
\[ \exp_2( \tau_{[1, 2r-1]}S^0 \wedge  \Sigma^\infty\mathbb{RP}^{2r} ) \geq r-1,  \]
so that 
$\exp_2( \tau_{[1, 2r-1]}S^0 ) \geq r-1$ as well (in view
of \Cref{smashexponent}). 

\end{proof}

Let $p$ be an odd prime. 
We will now give the analogous argument in this case.
\begin{proposition} 
\label{lowerboundp}
We have 
$\exp_p( \tau_{[1, n]} S^0) \geq \floor{ \frac{n - 1}{2p-2}} $.
\end{proposition} 
\begin{proof}
For simplicity, we will work with $B \Sigma_p$ (which implicitly will be
$p$-localized) rather
than $B \mathbb{Z}/p$. The $p$-local homology of $B \Sigma_p$ is well-known
(see \cite[Lem. 1.4]{maysteenrod} for the mod $p$ homology from which this
can be derived, together with the absence of higher Bocksteins): one has 
\[ H_i(B \Sigma_p; \mathbb{Z}_{(p)})  \simeq 
\begin{cases} 
\mathbb{Z}_{(p)} & i = 0 \\
\mathbb{Z}/p &  i = k(2p-2) - 1 , \ k > 0 \\
0 & \text{otherwise}
 \end{cases}. 
\]
One can thus build a cell decomposition of the (reduced) suspension spectrum
$\Sigma^\infty B \Sigma_p$ with cells in
degrees $\equiv 0, -1 \mod (2p-2)$ starting in degree $2p-1$. 

Let $k  > 0$, and consider the
$((2p-2)k) $-skeleton of this complex. We obtain a finite $p$-torsion spectrum $Y_k$
equipped with a map 
\[ Y_k \to \Sigma^\infty B \Sigma_p  \]
inducing an isomorphism in $H_*(\cdot; \mathbb{Z}_{(p)})$ up to and
including degree
$k(2p-2)$. 
That is, by universal coefficients, $H^i(Y_k; \mathbb{Z}_{(p)}) \simeq \mathbb{Z}/p$
if $i = 2p-2, 2(2p-2), \dots, k(2p-2)$ and is zero otherwise.

We now claim 
\begin{equation}\label{K0Yk} K^0(Y_k) \simeq \mathbb{Z}/p^k . \end{equation}
In order to see this, we use the Atiyah-Hirzebruch spectral sequence (AHSS) $$H^*(Y_k;
\mathbb{Z}) \implies K^*(Y_k).$$ Since the cohomology of $Y_k$ is concentrated
in even degrees, the AHSS degenerates and we find that $K^0(Y_k)$ is a
finite $p$-group of length $k$. However, the extension problems are
solved by naturality with the map $Y_k \to \Sigma^\infty B \Sigma_{p}$, as
$\widetilde{K}^0( B \Sigma_p) \simeq \mathbb{Z}_p$ after $p$-adic completion.

Now $Y_k$ is  a finite spectrum with cells in degrees
$[(2p-2)-1, (2p-2) k ]$. 
Let $m = (2p-2)(k-1) + 1$. 
Then we have, by \Cref{sufficestotruncate} and \eqref{K0Yk},
\begin{equation} \label{expY_k} \exp_p( Y_k) = \exp_p( \tau_{[0, 
m]
} S^0 \wedge Y_k) \geq k.\end{equation}
Finally, $\exp_p( H \mathbb{Z} \wedge Y_k) = 1$ since the $p$-local homology
of $Y_k$ is annihilated by $p$. 
It follows that 
$\exp_p( \tau_{[1, 
m]} S^0) \geq k-1$, which 
is the estimate we wanted if we choose $k$  as large as possible so that $m = (2p-2) (k-1) + 1 \leq
n$. 
\end{proof}

\begin{remark} 
In view of the Kahn-Priddy theorem \cite{KPthm}, it is not surprising that the skeleta of
classifying spaces of symmetric groups should yield strong lower bounds for torsion in the
Postnikov sections of the sphere. 
\end{remark} 

\section{The Hurewicz map}

We next apply our results about the Postnikov sections
$\tau_{[1, m]} S^0$ to the original question of understanding the exponents in
the Hurewicz map. 
Let $Y$ be a connective spectrum. Then the Hurewicz map is realized as the map
in homotopy groups 
induced by the map of spectra
\[ Y  \wedge S^0 \to Y \wedge H \mathbb{Z},  \]
whose fiber is $Y \wedge \tau_{[1, \infty]} S^0$. 
As a result of the long exact sequence in homotopy, we find:

\begin{proposition} 
\label{Hurconnect}
Let $Y$ be any connective spectrum.
\begin{enumerate}[(a)]
\item 
Suppose $\tau_{[1, n]} S^0$ is 
 annihilated by $N$ for some $N > 0$. Then 
any element $x$ in the kernel of the Hurewicz map $\pi_n(Y) \to H_n(Y; \mathbb{Z})$
satisfies $Nx = 0$.
\item 
Suppose $\tau_{[1, n-1]} S^0$ is 
annihilated by $N'$ for some $N' > 0$. Then 
for any element $y \in  H_n(Y; \mathbb{Z})$, $N'y $ is in the image of the
Hurewicz map.
\end{enumerate}
\end{proposition} 

The homotopy groups of $X \wedge \tau_{\geq 1} S^0$ are classically denoted
$\Gamma_i(X)$ (and called Whitehead's $\Gamma$-groups).
The following argument also appears in, for example,  \cite[Th. 6.6]{ArlALGK},
\cite[Cor. 4.6]{Scherer},
and \cite{beilinson}.
\begin{proof} 
For the first claim, consider the fiber sequence $Y \wedge 
\tau_{[1, \infty]} S^0 \to Y \to Y \wedge H \mathbb{Z}$. Any element $x \in
\pi_n(Y)$ in the kernel of the Hurewicz map lifts to an element $x' \in \pi_n(Y \wedge 
\tau_{[1, \infty]} S^0)$. It suffices to show that $Nx' = 0$. But we have an
isomorphism
\[ \pi_n(Y \wedge \tau_{[1, \infty]} S^0) \simeq \pi_n ( Y \wedge \tau_{[1,
n]} S^0), \]
and the latter group is annihilated by $N$ by hypothesis (and
\Cref{cohomfunctor}), so that $Nx' = 0$ as
desired. 

Now fix $y \in H_n(Y; \mathbb{Z})$. In order to show that $N'y$ belongs to the
image of the Hurewicz map, it suffices to show that it maps to zero via the
connective homomorphism into $\pi_{n-1}( Y \wedge \tau_{[1, \infty]} S^0)$. But
we have an isomorphism 
$\pi_{n-1}( Y \wedge \tau_{[1, \infty]} S^0) 
\simeq 
\pi_{n-1}( Y \wedge \tau_{[1, n-1]} S^0)
$
and this latter group is annihilated by $N'$. 
\end{proof} 

\begin{remark} 
One has an evident $p$-local version of \Cref{Hurconnect}
for $p$-local spectra if one works instead with $\tau_{[1, n]} S^0_{(p)}$.

\end{remark} 

\begin{proof}[Proof of \Cref{ourHurbound}] 
The main result on exponents follows now by combining \Cref{Hurconnect} and our
upper bound estimates in \Cref{maintheorem}.
\end{proof} 
It remains to show that the bound is close to being the best possible. 
This will follow by re-examining our arguments for the lower bounds.

\begin{proof}[Proof of \Cref{hurbestpossible}] 

We start with the prime $2$. 
For this, we use the space $\mathbb{RP}^{2k}$ and form the endomorphism ring
spectrum $Z = \hom(\Sigma^\infty\mathbb{RP}^{2k}, \Sigma^\infty\mathbb{RP}^{2k})
\simeq \Sigma^\infty\mathbb{RP}^k \wedge \mathbb{D} ( \Sigma^\infty
\mathbb{RP}^{2k})$ where $\mathbb{D}$ denotes Spanier-Whitehead duality. 
The spectrum $Z$ is not connective, but it is $(1-2k)$-connective (i.e., its cells begin
in degree $1-2k$).
Then we have a class $x \in \pi_0(Z)$ representing the identity self-map of
$\Sigma^\infty \mathbb{RP}^{2k}$. We know that $x$ has order at least $2^k$
(in view of \eqref{expRP}),
but that $2x$ maps to zero under the Hurewicz map since the homology of $Z$ is a
sum of copies of $\mathbb{Z}/2$ in various degrees by the integral K\"unneth
formula and since the homology of $\mathbb{RP}^{2k}$ is annihilated by $2$. If we replace $Z$ by
$\Sigma^{2k-1} Z$, we obtain a connective spectrum together with a class (the
translate of $2x$) in
$\pi_{2k-1}$ of order at least $2^{k-1}$ which maps to zero 
under the Hurewicz map. 

At an odd prime, one carries out the analogous procedure 
using the spectra $Y_k$ used in \Cref{lowerboundp}, and \eqref{expY_k}. One
takes $k = r+1$. 
\end{proof}

\begin{remark} 
We are grateful to Peter May for pointing out the following. Choose $q \geq 0$, and 
consider the cofiber sequence
\[ C = \tau_{\geq 0} S^{-q} \to S^{-q} \to \tau_{<0} S^{-q}.  \]
Choosing $n > 0$ and $q$ appropriately, we can find an element in $\pi_n( C) =
\pi_{n+q}(S^0)$ of large exponent (e.g., using the image of the
$J$-homomorphism), larger than $\mathrm{exp}(\tau_{[1, n]}S^0)$. This element must therefore \emph{not} be annihilated by the
Hurewicz map $\pi_n(C) \to H_n(C; \mathbb{Z})$. Let the image in $H_n(C;
\mathbb{Z})$ be $x$. However, the map $H_n(C;
\mathbb{Z}) \to H_n(S^{-q}; \mathbb{Z})$ is zero, so 
$x$ must be in the image of $H_{n+1}(\tau_{<0} S^{-q}; \mathbb{Z})$. 
This gives interesting and somewhat mysterious examples of homology classes in  
degree $n$ of a \emph{coconnective} spectrum. 
\end{remark} 

\section{Applications} %Proof of \Cref{torsionorder}}

We close the paper by noting a few applications of 
considering the exponent of the spectrum itself. These are mostly formal and independent of
\Cref{maintheorem}, which however then supplies the explicit bounds. 

We begin by recovering and improving upon a result from \cite{arlettazkinv} on
$k$-invariants.
\begin{theorem} 
\label{vanishkinv}
Let $X$ be any connective spectrum. Then the $n$th $k$-invariant $\tau_{\leq
n-1} X \to \Sigma^{n+1} H \pi_n X$ is annihilated by $\mathrm{exp} ( \tau_{[1, n]}
S^0)$.
\end{theorem} 
\begin{proof} 
It suffices to show that $H^{n+1}( \tau_{\leq n-1} X; \pi_n X)$ is
annihilated by 
$\mathrm{exp} ( \tau_{[1, n]}
S^0)$. By the universal coefficient theorem (and the fact that the universal
coefficient exact sequence splits), it suffices to show that 
the two abelian groups $H_n(\tau_{\leq n-1} X; \mathbb{Z})$ and 
$H_{n+1}(\tau_{\leq n-1} X; \mathbb{Z})$ are each annihilated by $\mathrm{exp} ( \tau_{[1, n]}
S^0)$. 
This follows from the cokernel part of \Cref{Hurconnect} 
because $\tau_{\leq n-1} X$ has 
no homotopy in degrees $n$ or $n+1$. 
\end{proof}

\begin{corollary} 
\label{ourboundkinv}
If $X $ is a connective spectrum, then the $n$th
$k$-invariant of $X$ has $p$-exponent at most (for $p=2$) $\ceil{\frac{n}{2}} +
3$ or (for $p> 2$) $\ceil{\frac{n+3}{2p-2}} +1$. 
\end{corollary} 
Asymptotically, 
\Cref{ourboundkinv} is stronger than the results of \cite{arlettazkinv}, which give
$p$-exponent $n  - C_p$ for $C_p$ a constant depending on $p$, as $n \to
\infty$.

Next, we consider a question about the homology of infinite loop \emph{spaces.}

\begin{theorem} 
\label{hurmapifl}
Let $X$ be an $(m-1)$-connected infinite loop space. Then the kernel of the
(unstable) Hurewicz map 
$\pi_n(X) \to H_n(X; \mathbb{Z})$ is annihilated by $\mathrm{exp}(\tau_{[1, n-m]}
S^0)$. Therefore, the $p$-exponent of the kernel is at most 
(for $p=2$) $\ceil{\frac{n-m}{2}} +
3$ or (for $p> 2$) $\ceil{\frac{n-m+3}{2p-2}} +1$. 
\end{theorem} 
This improves upon (and makes explicit) a result of Beilinson \cite{beilinson},
who also considers the cokernel of the map from $\pi_n(X)$ to the
\emph{primitives}
in $H_n(X; \mathbb{Z})$.
\begin{proof} 
Without loss of generality, we can assume that $X$ is $n$-truncated.
Let $Y$ be the $m$-connective spectrum that deloops $X$. 
Consider the cofiber sequence
\[   Y \to \tau_{\leq n-1} Y \to \Sigma^{n+1} H \pi_n Y. \]
By \Cref{vanishkinv}, the $k$-invariant map 
$\tau_{\leq n-1} Y \to \Sigma^{n+1}H \pi_n Y$ is annihilated by $
\mathrm{exp}(\tau_{[1, n-m]}
S^0)$. Consider the rotated cofiber sequence
\[ \Sigma^{-1} \tau_{\leq n-1}Y \to  \Sigma^n H \pi_n Y\to Y.  \]
Using the natural long exact sequence, we obtain that there exists a map 
\[ Y \to  \Sigma^n H\pi_n Y  \]
which induces multiplication by 
$\mathrm{exp}(\tau_{[1, n-m]}
S^0)$ on $\pi_n$. Compare \cite[Lem. 4]{ArlSLn} for this argument.

Delooping, we obtain a map of spaces $\phi \colon X \to K(\pi_n X, n)$ which induces
multiplication by 
$\mathrm{exp}(\tau_{[1, n-m]}
S^0)$ on $\pi_n$. Now we consider the commutative diagram
\[ \xymatrix{
\pi_n(X) \ar[d]^{\phi_*}  \ar[r] &  H_n(X; \mathbb{Z}) \ar[d]^{\phi_*} \\
\pi_n ( K(\pi_n X, n)) \ar[r]^{\simeq} &  H_n( K(\pi_n X, n); \mathbb{Z})
}.\]
Choose $x \in \pi_n(X)$ which is in the kernel of the Hurewicz map; the diagram
shows that $\phi_*(x) = 
\mathrm{exp}(\tau_{[1, n-m]}
S^0)
x = 0$, as desired.
\end{proof} 

Next, we give a more careful statement (in terms of exponents of Postnikov
sections of $S^0$) of \Cref{torsionorder}, and prove it. 
Note that this result is generally much sharper than \Cref{firstbound}. 
\begin{proposition} \label{torsionorder2}
Let $X$ be a $p$-torsion spectrum 
with homotopy groups concentrated in an interval $[a,b]$ of length $\ell =
b-a$. Suppose $p^k$ annihilates $\pi_i(X)$ for each $i$.
Then $\exp_p(X) \leq k + \exp_p(\tau_{[1, \ell]} S^0) + \exp_p(\tau_{[1,
\ell-1]} S^0) = k + \frac{\ell}{p-1} + O(1)$.
\end{proposition} 

The argument is completely formal except for the equality
$ \exp_p(\tau_{[1, \ell]} S^0) + \exp_p(\tau_{[1,
\ell-1]} S^0) =  \frac{\ell}{p-1} + O(1)$. This comparison is a
consequence of \Cref{maintheorem}. 
\Cref{torsionorder2} plus the estimates of \Cref{maintheorem} yield
\Cref{torsionorder}. We note that a simple calculation can make $O(1)$ explicit. 
\begin{proof}
Without loss of generality, we assume $a = 0$ so $b  = \ell$.
We consider the cofiber sequence
and diagram
\[
\tau_{[1, \infty]} S^0 \wedge X \to   X \to  H \mathbb{Z} \wedge
X   
.\]
This induces an exact sequence
\begin{equation} \label{bighalfexact} \pi_0\hom(H \mathbb{Z} \wedge X, X) \to
\pi_0 \hom(X, X) \to \pi_0\hom( \tau_{[1,
\infty]} S^0 \wedge X, X).   \end{equation}

Let $R_1 = p^{\exp_p(\tau_{[1, b]} S^0)}, R_2
=  p^{\exp_p(\tau_{[1, b-1]} S^0)}$. 
We will bound 
the exponents of the terms on either side by $R_1$ and $R_2 p^k$ to bound the exponent on the group in
the middle (which will give a torsion exponent for $X$).
Note that  since $X$ is concentrated in degrees $[0, b]$,
one has
\begin{gather}
\pi_0\hom(H \mathbb{Z} \wedge X, X) \simeq
\pi_0\hom(\tau_{\leq b}( H \mathbb{Z} \wedge X), X) \label{c1} \\ 
\pi_0\hom( \tau_{[1,
\infty]} S^0 \wedge X, X) \simeq \pi_0\hom( \tau_{[1, b]} S^0 \wedge X, X).
\label{c2}
\end{gather}

We claim first that $\tau_{\leq b}(H \mathbb{Z} \wedge X)$ is annihilated by $
R_2 p^k$. 
To see this, it suffices, since $\tau_{\leq b}(H \mathbb{Z} \wedge X)$ is a generalized
Eilenberg-MacLane spectrum, to show that its homotopy groups are each
annihilated by $R_2p^k$. 
That is, we need to show that each of the homology groups of $X$ is
annihilated by $R_2p^k$. 
For this, we consider the Hurewicz homomorphism
\[ \pi_i(X) \to H_i(X; \mathbb{Z}), \quad i \leq b.  \]
The source is annihilated by $p^k$, and \Cref{Hurconnect} implies that the cokernel is
annihilated by $R_2$. This proves that $H_i(X; \mathbb{Z})$ is annihilated by
$R_2p^k$ for each $i \in [0, b]$.
Therefore, \eqref{c1} is annihilated by $R_2 p^k$.

Next, we claim that $\tau_{[1, b]} S^0 \wedge X$ is annihilated by $R_1$. This
is evident by \Cref{smashexponent},
because $\tau_{[1, b]} S^0$ is. 
Thus, \eqref{c2} is annihilated by $R_1$.

Putting everything together, we obtain the desired torsion bounds on the ends of
\eqref{bighalfexact}, so that the middle term is annihilated by $R_1 R_2 p^k$, and we are done. 
\end{proof}

Finally, we show that our results have applications to exponent theorems in
equivariant stable homotopy theory. 
We begin by noting a useful example on the stable homotopy of classifying
spaces.
\begin{example} 
\label{piBG}
Let $G$ be  a finite group and let $\Sigma^\infty BG$ be the reduced suspension spectrum of
the classiying $BG$. Then for any $n$, the abelian group $\pi_n(\Sigma^\infty
BG)$ is annihilated by $|G| \mathrm{exp}(\tau_{[1, n]} S^0 ) $. This follows
from \Cref{Hurconnect}, since the integral homology of $BG$ is annihilated by
$|G|$. In fact, we obtain that the spectrum $\tau_{[1, n]} BG$ is annihilated by $|G|
\mathrm{exp}(\tau_{[1, n]} S^0 )$. We do not know if the growth rate of
$\mathrm{exp}( \tau_{[1, n]} BG)$ is in general comparable to this. 
\end{example}

\newcommand{\Gsp}{\mathcal{S}_G}

Let $G$ be a finite group, and consider the  
homotopy theory $\Gsp$ of genuine $G$-equivariant spectra. 
The symmetric monoidal category $\Gsp$ has a unit object, the equivariant sphere $S^0$. We will be interested
in exponents for the equivariant stable stems $\pi_{n, G}(S^0) = \pi_0 \hom_{\Gsp}(S^n,
S^0)$. 
More generally, we  will replace the target $S^0$ by a representation sphere
$S^V$, for $V$ a finite-dimensional real representation 
of $G$. 
In this case, we will write $\pi_{n, G}(S^V) = \hom_{\Gsp}(S^n, S^V)$.
For a subgroup $H \subset G$, we will write
$WH = N_G(H)/H$ for the Weyl group.
\begin{theorem} 
\label{equivexponent}
Let $V$ be a finite-dimensional $G$-representation.
Suppose $n$ is not equal to the dimension $\dim V^H$ for any subgroup $H
\subset G$. 
Then 
the abelian group $\pi_{n, G}(S^V)$ is annihilated by the least
common multiple of $\{|WH|\mathrm{exp} (\tau_{[1, n - \dim V^H]}
S^0)\}$ as $H \subset G$ ranges over all the subgroups with $\dim V^H < n$.
In particular, the $p$-exponent of $\pi_{n, G}(S^V)$ is at most 
\begin{align*} \exp_p( \pi_{n, G}(S^V)) & \leq  \max_{H \subset G, \dim V^H < n}
\left( v_p(|WH|)+   \exp_p( \tau_{[1, n - \dim V^H]} 
S^0) \right)  \\ 
& =  \max_{H, \dim V^H < n}\left( v_p(|WH|) + \frac{n - \dim V^H}{2p-2}
\right) + O(1) ,  \end{align*}
where $v_p$ denotes the $p$-adic valuation. 
\end{theorem} 

\begin{remark}
When $n > \dim V$, 
the least common 
multiple simplifies to $|G| \mathrm{exp}( \tau_{[1, n - \dim V]} S^0)$. 
\end{remark}

\begin{proof} 
This follows from the Segal-tom Dieck splitting \cite{tomDieck}, which implies that 
\[ \pi_{n  , G}(S^V)  = \bigoplus_{H }\pi_n \left( (\Sigma^\infty 
S^{V^H}
)
_{h W_H}
\right),   \]
where $H$ ranges over a system of conjugacy classes of subgroups of $G$. When $V$ is the trivial representation, we can apply
\Cref{piBG} to conclude.  

In general, $(\Sigma^\infty 
S^{V^H}
)
_{h W_H}$ is $\dim V^H $-connective. 
Moreover, the homology $H_*(S^{V^H}; \mathbb{Z})$ is concentrated in dimension $\dim V^H$, so that
it follows that for $n > \dim V^H$, 
$H_n( (\Sigma^\infty S^{V^H})_{h WH}; \mathbb{Z})$ is annihilated by the order
of $WH$.
For $n < \dim V^H$, there is no contribution in homotopy from  
$(\Sigma^\infty 
S^{V^H}
)
_{h W_H}$. 
Applying
\Cref{Hurconnect} and \Cref{maintheorem}, we obtain the desired exponent result. 
\end{proof} 

In equivariant stable homotopy theory, one is more generally interested in maps
$S^W \to S^V$ where $W, V$ are orthogonal representations of $G$.
Unfortunately, the method of \Cref{equivexponent} does not seem to give
anything, unless $W$ is very small relative to $V$, in which case one can
use a cell decomposition of $S^W$ and apply \Cref{equivexponent} to the
individual cells. 

\bibliographystyle{alpha}
\bibliography{torsion}
\end{document}